\documentclass[a4paper]{amsart}

\newtheorem{thm}{Theorem}
\newtheorem{lem}{Lemma}

\theoremstyle{definition}
\newtheorem{defi}{Definition}
\newtheorem{rem}{Remark}

\usepackage{todonotes}

\begin{document}

\title[CW complexes and arithmetic matroids]{Colorings and flows on CW complexes,\\ Tutte quasi-polynomials and arithmetic matroids}
\author{Emanuele Delucchi \and Luca Moci}

\maketitle

\begin{abstract}

In this note we provide a higher-dimensional analogue of Tutte's theorem on colorings and flows of graphs, by showing that the theory of arithmetic Tutte polynomials and quasi-polynomials encompasses invariants defined for CW complexes by Beck--Breuer--Godkin--Martin \cite{BBGM} and Duval--Klivans--Martin  \cite{DKM}.  Furthermore, we answer a question by Bajo--Burdick--Chmutov \cite{BBC}, concerning the modified Tutte--Krushkal--Renhardy polynomials defined by these authors: to this end, we prove that the product of two arithmetic multiplicity functions on a matroid is again an arithmetic multiplicity function.

\end{abstract}

%\tableofcontents

\section{Introduction}

The enumeration of colorings, flows and spanning trees on graphs are classical topics, unified by a two-variable polynomial due to W.\ T.\ Tutte \cite{Tutte}. This polynomial specializes to both the coloring counting and the flow counting functions, and it evaluates to the number of spanning trees. H.\ Crapo extended Tutte's definition 
to arbitrary matroids and since then this {\em Tutte polynomial}
 went on to become one of the most studied matroid polynomial invariants with great theoretical significance and a host of applications --- e.g.,\ in statistics and physics. Recently, this classical setup has been generalized  in two ways.

First, the concept of coloring and flow has been generalized from graphs to higher dimensional objects such as simplicial complexes by Beck and Kemper \cite{BK} and, more generally, to CW complexes by Beck--Breuer--Godkin--Martin \cite{BBGM} and Duval--Klivans--Martin \cite{DKM}. These authors showed, among other things, that the  functions counting the number of colorings and flows with $q$ values on a CW complex is a quasi-polynomial in $q$. In a related vein,  Bajo, Burdick and Chmutov \cite{BBC} introduced a family of \emph{modified TKR polynomials} that connects Kalai's enumeration of weighted cellular spanning trees of complexes \cite{Ka} to a class of polynomials defined by Krushkal and Renhardy \cite{KR} in their study of graph embeddings and to a polynomial defined by Bott \cite{Bo}.

On the other hand, in collaboration with M.~D'Adderio \cite{Md1} and with P.~Br\"and\'en \cite{BM} the second-named author developed a theory of arithmetic matroids as ``matroids decorated with a multiplicity function'', abstracting the arithmetic properties of lists of elements in finitely generated abelian groups. To each arithmetic matroid is naturally associated an {\em arithmetic Tutte polynomial}. These polynomials have been in the focus of recent and lively research, which brought to light manifold connections and a rich structure theory. For instance, arithmetic Tutte polynomials specialize to Poincar\'e polynomials of toric arrangements \cite{M1}, to Ehrhart polynomials of zonotopes \cite{Md3} and to the Hilbert series of some zonotopal spaces \cite{Le}. Moreover, they can be recovered from the Tutte polynomials for group actions on semimatroids \cite{DR}, and they satisfy a convolution formula \cite{BaLe}. 

With a list of elements in a finitely generated abelian group is also associated a \emph{Tutte quasi-polynomial} \cite{BM}, which interpolates between the (ordinary) Tutte polynomial and the arithmetic Tutte polynomial. This quasi-polynomial does not depend only on the arithmetic matroid, but on a finer structure: \emph{a matroid over $\mathbb Z$} in the sense of \cite{FM}.
As pointed out in \cite{BBGM}, the enumerating functions of colorings and flows on a CW complex are not matroidal, and hence cannot be obtained from the ordinary Tutte polynomial. In this paper we show that, however, they are specializations of the Tutte quasi-polynomial, thus providing a higher-dimensional analogy to Tutte's celebrated theorem for graphs \cite{Tutte}.

Moreover, we show that the set of arithmetic matroids over a fixed underlying matroid has a natural structure of commutative monoid. This implies that the modified TKR polynomials are indeed arithmetic Tutte polynomials; in particular, their coefficients are positive. 

In this way we address questions of the authors of \cite{BBC,BBGM,DKM}, who ask whether and how the coloring and flow polynomials for CW complexes and the modified TKR polynomials are related to arithmetic matroids. 

%We start with 
\subsection*{Structure of the paper}
In Section \ref{sec:A} we start off with some preliminaries on incidence algebras and arithmetic matroids. We prove a general theorem about products of integer functions on posets (Theorem \ref{thmnuovo}) and specialize it to one about products of arithmetic multiplicity functions (Theorem \ref{Thm1}).

In Section \ref{sec:B} we discuss how the flow and chromatic quasi-polynomials for CW complexes are indeed specializations of the Tutte quasi-polynomial (Theorem \ref{Thm2}). 

We close with Section \ref{sec:C} where we prove that the modified TKR polynomial is the arithmetic Tutte polynomial of an arithmetic matroid (Theorem \ref{Thm3}).

\subsection*{Acknowledgements} The first-named author has been supported by the Swiss National Foundation Professorship grant PP00P2\_150552/1. We thank Yvonne Kemper for pointing out \cite{BL}, and Fengwei Zhou for finding an error in a previous version of this paper.

\section{On arithmetic matroids}\label{sec:A}
\subsection{Poset theory preliminaries}
The goal of this section is to prove a result on M\"obius functions of posets (short for ``partially ordered sets'') which will serve as a stepping stone towards Theorem \ref{Thm1}. We will assume familiarity with basic terminology of poset theory. We refer the reader unfamiliar with it to \cite{Sta}.

Throughout, we will let $P$ denote a finite poset.\footnote{This will avoid unnecessary technicalities and will suffice for the applications later in the paper, even though most of what we will prove in this section holds in the generality of locally finite posets.} An {\em interval} of $P$ is any subset of $P$ of
the form $[x,y]:=\{z\in P \mid x\leq z \leq y\}$ for some $x,y\in P$, $x\leq
y$. The set of intervals of $P$ is denoted $I(P)$.  
% as follows.
%$$
%(f_1 \ast f_2)(x,y):=\sum_{x\leq z \leq y} %f_1(x,z)f_2(z,y)
%$$

The so-called {\em M\"obius
  function} of $P$ is the function 
  $$\mu: I(P) \to \mathbb Z$$ defined recursively as follows
  $$
  \left\{\begin{array}{ll}
  \displaystyle{\mu(p,p) = 1} & \textrm{for all }p\in P,\\[4pt]
  \displaystyle{\sum_{p_1\leq q \leq p_2} \mu(p_1,q) = 0} & \textrm{for all } p_1 < p_2 \textrm{ in }P, 
  \end{array}\right.
  $$
  where for simplicity we write $\mu(x,y):=\mu([x,y])$.

 The {\em (dual) M\"obius transform}\footnote{We will henceforth simply use the term {\em M\"obius transform}. It is referred to as ``dual form'' in \cite[Proposition 3.7.2]{Sta}} of a function $m: P\to R$ is the function
 \begin{align*}
 m^\mu : P\to & R \\
  p\mapsto & \sum_{q\geq p} \mu(p,q) m(q). 
 \end{align*}
 It is characterized by $m(p) = \sum_{q\geq p} m^{\mu}(q)$.

Consider two elements $p,p' \in P$. If there is an element $x\in P$ with 
$$
\{q\in P\mid q\geq p,\,\, q\geq p'\} = \{q\in P \mid q\geq x\}
$$ 
then $x$ is unique, called the {\em  meet} (or minimal upper bound) of $p$ and $p'$, and denoted by $p\vee p'$. If every pair $p,p'\in P$ admits a meet, the poset $P$ is called a {\em meet semilattice}.

The following lemma should be folklore. We give here a proof for completeness, because we do not know of a reference for it.

\begin{lem}\label{newlemma} Let $P$ be 
%an atomic lattice
a meet-semilattice, and $D: P\to \operatorname{Sets}$ be a function such that $D(p)\cap D(q) = D(p\vee q)$ for all $p,q\in P$.
Then, 
$$\sum_{q\geq p} \mu(p,q) \vert D(q) \vert \geq 0$$ 
for all $p\in P$.
\end{lem}
\begin{proof} 
Define for all $p\in P$ 
$$
G(p):= D(p) \setminus \bigcup_{q> p} D(q),\quad\quad
f(p):=\vert G(p) \vert \geq 0.
$$
We claim that $$D(p)=\biguplus_{q\geq p} G(q).$$ 
%and the union is disjoint. 
The right-to left inclusion is clear:  
%by Remark \ref{joins}, since $P$ is atomic, 
$q \geq p$ means $q=p\vee q$, 
%for some $p'\in P$, 
hence  $G(q)\subseteq D(q)=D(p)\cap D(q) \subseteq D(p)$. For the left-to-right inclusion consider $x\in D(p)$. The set $P_x=\{q\in P \mid x\in D(q) \}$ has a unique maximal element $\hat p$ (since $x\in D(q)$ and $x\in D(q')$ imply $x\in D(q\vee q')$ -- hence, $q,q'\in P_x$ imply $q\vee q'\in P_x$). Now we see that $x\in D(\hat p) \setminus \bigcup_{q>\hat p} D(q)=G(\hat p)$. Uniqueness of $\hat p$ implies that the union is indeed disjoint.

Thus, for all $p\in P$ we have $$\vert D(p) \vert = \sum_{q\geq p} f(q)$$
and, by M\"obius inversion,

$$
\sum_{q\geq p}\mu(p,q)\vert D(q) \vert = f(p) \geq 0
$$
as required.
\end{proof}

\begin{thm}\label{thmnuovo}
Let $P$ be a meet-semilattice, and consider two functions $m_1,m_2:P\mapsto Z$. If $(m_1)^\mu(p) \geq 0$ and $(m_1)^\mu(p) \geq 0$ for all $p\in P$, then $(m_1m_2)^\mu (p) \geq 0$ for all $p\in P$.
\end{thm}
\begin{proof}
The positivity hypothesis allows us to define, for every $i=1,2$ and $p\in P$, a set 
$$
G_i(p):=\{X_1^{i,p},\ldots ,X_{m_i^\mu(p)}^{i,p}\},
$$
where the $X^{i,p}_j$ are pairwise distinct formal elements -- i.e., 
$X^{i,p}_j = X^{i',p'}_{j'}$ if and only if $i=i'$, $p=p'$, $j=j'$. Then, set
$$
A_i(p):=\biguplus_{q\geq p} G_i(q).
$$
%the union being clearly disjoint. 
Then, 
$$A_i(p') \cap A_i(p'') = A_i(p'\vee p'')$$
 Notice also that, by definition of $m_i^\mu$,
$$
\vert A_i(p) \vert = \sum_{q\geq p} m_i^\mu(q) = m_i(p).
$$

Consider now the family of sets $(A_{12}(p))_{p\in P}$ defined by
$$
A_{12}(p):= A_1(p) \times A_2(p).
$$

Since cartesian products commute with intersections, for $p',p''\in P$ we have 
$$
A_{12}(p')\cap A_{12}(p'') = 
(A_1(p')\cap A_1(p''))\times (A_2(p')\cap A_2(p'')) = A_{12}(p'\vee p'')
$$
and thus, by Lemma \ref{newlemma}, $$
\sum_{q\geq p} \mu(p,q) \vert A_{12}(q) \vert \geq 0.
$$
The claim now follows because $\vert A_{12}(q) \vert = m_1(q)m_2(q)$ for all $q\in P$.

\end{proof}

\subsection{Arithmetic matroids}

In this section we recall basic definitions on matroids and arithmetic matroids in order to set some notation, and we prove Theorem \ref{Thm1}. For background on matroid theory we refer, e.g., to Oxley's textbook \cite{Oxley}, while our presentation of arithmetic matroids follows mostly \cite{BM}.
 
\newcommand{\rk}{\operatorname{rk}}
\begin{defi}
  A {\em matroid} is given by a pair $(E,\rk)$, where $E$ is a finite
  set and $\rk : 2^E \to \mathbb N$ is a function such that, for all $X,Y\subseteq E$,
  \begin{itemize}
  \item[(R1)] $\rk(X)\leq \vert X \vert$,
  \item[(R2)] $X\subseteq Y$ implies $\rk(X)\leq \rk(Y)$,
  \item[(R3)] $\rk(X\cup Y) + \rk(X\cap Y) \leq \rk(X) + \rk (Y)$.
  \end{itemize}

  A {\em molecule} in a matroid is a triple $\alpha:=(R,F,T)$ of
  disjoint subsets of $E$ such that, for every $A\subseteq E $ with $R
  \subseteq A \subseteq R\cup F\cup
  T$, 
  $$\rk(A) = \rk(R) + \vert A\cap F \vert.$$
  To the molecule $\alpha$, following e.g.\ \cite{DR}, we associate a poset $$B_{\alpha}=\{(T',F')
  \mid T'\subseteq T, F'\subseteq F\}$$ ordered by $(T',F')\leq (T'',F'')$ if $T'\subseteq T''$, $F'\supseteq
F''$.
\end{defi}

\begin{rem} The poset $B_\alpha$ is bounded, with unique minimal element $(\emptyset, F)$ and unique maximal element $(T,\emptyset)$. Moreover, every interval in $B_\alpha$ (say, $[(F',T'),(F'',T'')]$) is the poset $B_{\alpha'}$
  for another molecule (i.e., $\alpha'=(R\cup F''\cup T', F',T''\setminus T')$).
\end{rem}

Given any function $m: 2^E \to \mathbb Z$ and a molecule $\alpha=(R,F,T)$ of a
matroid over the ground set $E$, we define  $m_\alpha : B_\alpha \to
\mathbb Z$ as the function with
$$
m_\alpha (F',T') := m(R\cup F' \cup T').
$$ 

\begin{defi}
  An {\em arithmetic matroid} is a triple $(E,\rk,m)$ where $(E,\rk)$
  is a matroid, and $m: 2^E \to \mathbb Z$ is a function satisfying
  the following axioms.
  \begin{itemize}
  \item[(P)] For every molecule $\alpha$ of $(E,\rk)$
$$(m_\alpha)^\mu (\hat 0) \geq 0.$$
 \item[(Q)]  For every molecule $\alpha=(R,F,T)$ of $(E,\rk)$
 $$m(R)m(R\cup F\cup T) = m(R\cup F) m(R\cup T).$$
  \item[(A)] For all $A\subseteq E$ and all $e\in E$,
$$
\left(\frac{m(A\cup e)}{m(A)}\right)^{2(\rk(A\cup e) - \rk(A)) -1} \in \mathbb Z.
$$
  \end{itemize}
\end{defi}

\begin{rem}
  Axiom (P) is usually given in a different form. In fact, for a
  molecule $\alpha=(R,F,T)$ we see that the poset $B_\alpha$ is
  boolean and the length of the interval $(B_\alpha)_{\leq (T',F')}$
  is $\vert  T' \vert + \vert F\setminus F' \vert$. Therefore, the
  M\"obius function of $B_{\alpha}$ satisfies 
$$
\mu(\hat 0, (T',F')) = (-1)^{\vert  T' \vert + \vert  F \setminus F' \vert}.
$$ 
If we now expand our form of axiom (P) we get
$$
(m_\alpha)^\mu(\hat 0)=
\sum_{
\substack{
T'\subseteq T\\
F'\subseteq F}
}
\mu(\hat 0, (T',F')) m(R\cup F'\cup T')
$$
$$
=\sum_{
\substack{
T'\subseteq T\\
F'\subseteq F}
}
(-1)^{\vert T' \vert + \vert F \setminus  F' \vert}
m(R\cup F'\cup T')
= (-1)^{\vert T\vert}\sum_{
\substack{
T'\subseteq T\\
F'\subseteq F}
}
(-1)^{\vert T\setminus T' \vert + \vert F\setminus F' \vert}
m(R\cup F'\cup T')
$$
$$
=(-1)^{T}
\sum_{R\subseteq A \subseteq R\cup F\cup T}
(-1)^{\vert (R\cup F\cup T) \setminus A \vert} m(A),
$$
and we recover the formulation given in \cite{BM}.
\end{rem}

\subsection{Product of multiplicity functions}
Consider now a fixed matroid $(E,\rk)$, 
two (possibly different) functions $m', m'' : 2^E \to \mathbb Z$ and
their (pointwise) product $m:= m'm''$

\begin{lem}\label{lem:P}
  If both $m'$ and $m''$ satisfy axiom (P), so does $m=m'm''$.
\end{lem}
\begin{proof}
  Suppose $m' $ and $m''$ both satisfy (P) and consider a molecule
  $\alpha$. The poset $B_\alpha$ is boolean, hence in particular a (meet semi-)lattice. Since every interval of $B_\alpha$ defines a molecule, $m'$ and $m''$ satisfy the conditions of Theorem \ref{thmnuovo} on $B_\alpha$. Hence, $(m'm'')^\mu (\hat 0)\geq 0$ 
  
\end{proof}

\begin{thm}\label{Thm1}
  If both $(E,\rk,m')$ and $(E,\rk,m'')$ are arithmetic matroids, then
  $(E,\rk,m'm'')$ is also an arithmetic matroid.
\end{thm}
\begin{proof}
  The triple $(E,\rk,m'm'')$ satisfies (P) by Lemma \ref{lem:P}, and
  (Q), (A) trivially.
\end{proof}

\begin{rem} This theorem endows the set of arithmetic matroids over a fixed underlying matroid with a natural product, which makes it into a commutative monoid. We leave the investigation of this algebraic structure as an open problem.
\end{rem}

\section{On Tutte quasi-polynomials associated to cell complexes}\label{sec:B}

\subsection{The Tutte quasi-polynomial}
Let $G$ be a finitely generated abelian group, $E$ be a finite set, and $\mathcal L=\{\{g_e: e\in E\}\}$ be a list (multiset) of elements in $G$.
For every $A\subseteq E$ we denote by  $\mathcal L_A$ the sublist $\{\{g_e : e\in A\}\}$, by $\langle\mathcal L_A\rangle$ the subgroup that it generates, and by $G_A:= \operatorname{tor}(G/\langle\mathcal L_A\rangle)$ the torsion subgroup of the quotient $G/\langle\mathcal L_A\rangle$.
In \cite[Section 7]{BM}, the \emph{Tutte quasi-polynomial} of $\mathcal L$ is defined as follows.

$$Q_{\mathcal L}(x,y):=\sum_{A\subseteq E} \frac{|G_A|}{|(x-1)(y-1)G_A|}(x-1)^{\rk E-\rk A}(y-1)^{|A|-\rk A}.$$

 \begin{rem} If for every $A\subseteq E$ the integer $k=(x-1)(y-1)$ is coprime with $|G_A|$, then $kG_A:=\{kg \vert g\in G_A\}$ equals $G_A$ and we get the ordinary \emph{Tutte polynomial} of the matroid of linear dependencies among elements of $\mathcal L$:
$$T_{\mathcal L}(x,y):=\sum_{A\subseteq E} (x-1)^{\rk E-\rk A}(y-1)^{|A|-\rk A}.$$
  On the other hand, when for every $A\subseteq E$ the integer $k$ is a multiple of $|G_A|$ we have that $kG_A$ is trivial and we obtain the \emph{arithmetic Tutte polynomial}:
$$M_{\mathcal L}(x,y):=\sum_{A\subseteq E} {|G_A|}(x-1)^{\rk E-\rk A}(y-1)^{|A|-\rk A}.$$
  
   Therefore $Q_{\mathcal L}(x,y)$ is a quasi-polynomial function that in some sense interpolates between these two polynomials. It appeared as a specialization of a multivariate "Fortuin--Kasteleyn quasi-polynomial". 

\end{rem} 

Now recall the following definitions.

\begin{defi}[{\cite[Section 7]{BM}}] Let $G$, $E$ and $\mathcal L$ be as above.

\begin{enumerate}

\item A \emph{proper $q$-coloring} is an element $c\in \operatorname{Hom}(G, \mathbb Z_q)$ such that $c(g_e)\neq 0$ for all $e\in E$.

\item A \emph{nowhere zero $q$-flow} is a function $\phi: E\longrightarrow \mathbb Z_q \setminus \{0\}$ such that\\ $\sum_{e\in E}\phi(e)g_e=0$ in $G/qG$.

\end{enumerate}

The number of proper $q$-colorings and the number of nowhere zero $q$-flows are denoted by $\chi_{\mathcal L}(q)$ and $\chi^*_{\mathcal L}(q)$ respectively.

\end{defi}

The following statement generalizes a result of \cite{Md2}.

\begin{lem}[{\cite[Theorem 9.1]{BM}}]\label{cf:qp}
$$\chi_{\mathcal L}(q)=(-1)^{\rk E} q^{\rk G-\rk E}Q_{\mathcal L}(1-q, 0)$$

$$\chi^*_{\mathcal L}(q)=(-1)^{|E|-\rk E} \vert\operatorname{tor}(G)\vert^{-1}Q_{\mathcal L}(0, 1-q).$$
\end{lem}

In particular, $\chi_{\mathcal L}(q)$ and $\chi^*_{\mathcal L}(q)$ are quasi-polynomial functions of $q$, called the \emph{chromatic quasi-polynomial} and the \emph{flow quasi-polynomial} respectively.

\subsection{On flows and colorings on CW complexes}

Let $C$ be a CW complex of dimension $d$ and, for every $i=0,1,\dots, d$, let $C_i$ be the set of the $i$-dimensional cells of $C$. The top-dimensional boundary map $\partial: \mathbb Z^{C_{d}}\to \mathbb Z^{C_{d-1}}$ is represented by a matrix with integer entries, that (by a slight abuse of notation) we denote again by $\partial$. By reducing modulo $q$, we get a map $\overline{\partial}: \mathbb Z_q^{C_{d}}\to \mathbb Z_q^{C_{d-1}}$, that we can view as a matrix with coefficients in $\mathbb Z_q$.

%In \cite{BK} and \cite{BBGM} the following definitions are given:

\begin{defi}[{cf.\ \cite{BK} and \cite{BBGM}}] Let $C$ and $\partial$ be as above.

\begin{enumerate}
\item[(1')] a \emph{proper $q$-coloring} of $C$ is an element $c\in \mathbb Z_q^{C_{d-1}}$ such that all the entries of the vector $c \overline{\partial}$ are nonzero.

\item[(2')] a \emph{nowhere zero $q$-flow} on $C$ is an element $\phi\in \ker \overline{\partial}$ such that the coordinate $\phi(e)$ is nonzero for every $e\in C_d$.
\end{enumerate}
\end{defi}

The authors of \cite{BK} and \cite{BBGM} prove that the number of proper $q$-colorings and the number of nowhere zero $q$-flows are quasi-polynomial functions, that we will denote by $\chi_{C}(q)$ and $\chi^*_{C}(q)$.

In fact, to the (integer) matrix $\partial$ we can associate a Tutte quasi-polynomial, an arithmetic matroid and an arithmetic Tutte polynomial. 
With the following lemma we address \cite[Remark 3.15]{BBGM} by showing that the coloring- and flow- counting quasi-polynomials of \cite{BBGM} and \cite{DKM} are instances of the coloring- and flow- quasi-polynomials associated to the matrix $\partial$.

\begin{lem}\label{lem:eq}
Definitions (1') and (2') 
agree with 
definitions (1) and (2), when $G=\mathbb Z^{C_{d-1}}$, $E=\{1,2,\ldots |C_{d}|\}$, and $\mathcal{L}=\{\text{columns of }\partial\}$.
\end{lem}

\begin{proof}
Every $c\in\mathbb Z_q^{C_{d-1}}$ uniquely extends to a homomorphism $\tilde{c}\in \operatorname{Hom} (\mathbb Z^{C_{d-1}}, \mathbb Z_q)$. Then since  $\tilde{c}(g_e)=c\partial$, (1) specializes to (1').
On the other hand, definition (2') is equivalent to saying that $\phi$ is a function $C_d\to\mathbb Z_q\setminus\{0\}$ such that $\partial \phi=0$. This is precisely the specialization of definition (2).
\end{proof}

The following ``higher-dimensional analogue'' of Tutte's theorem about the dichromate \cite{Tutte} 
now follows immediately from  Lemma \ref{lem:eq} and Theorem \ref{cf:qp}.

\begin{thm}\label{Thm2}
With the notations above, we have:
$$\chi_{C}(q)=(-1)^{\rk \partial} q^{|C_{d-1}|-\rk \partial}Q_{\partial}(1-q, 0)$$

$$\chi^*_{C}(q)=(-1)^{|C_{d}|-\rk \partial} Q_{\partial}(0, 1-q).$$
\end{thm}

\begin{rem}\label{rem:cellmat}
As pointed out in \cite{BM}, the Tutte quasi-polynomial is not an invariant of the arithmetic matroid, but is an invariant of the \emph{matroid over $\mathbb{Z}$} associated to the matrix $\partial$. We call this matroid the \emph{cellular matroid} over $\mathbb{Z}$ of $C$.

\end{rem}

\begin{rem}\label{rem:cellmat2}
Underlying matroids of cellular matroids over $\mathbb Z$ (i.e., the matroids defined by the matrices $\partial$) have been studied in their own right. Allowing different generality for the complex $C$ one obtains different interesting classes of matroids. Already in the case where $C$ is a simplicial complex, the matroids obtained this way are strictly more general than graphical matroids \cite{BL}.
\end{rem}

\begin{rem} Given a $d$-dimensional CW-complex $C$, for every $j=0,1,\dots,d-1$ the $j$-skeleton of $C$ is itself a $j$-dimensional CW-complex $C^j$ for which we can carry out all considerations of this section. Thus, $C$ in fact gives rise to a class of arithmetic quasi-polynomials and arithmetic matroids. In the following section we will consider properties of this class as a whole.
\end{rem}

\section{On the modified Tutte-Krushkal-Renhardy polynomial}\label{sec:C}

When considering cell complexes as higher dimensional generalizations of graphs, besides flows and colorings it is natural to enumerate the analogue of  spanning trees. Following Kalai \cite{Ka}, this enumeration is weighted by the cardinality of the torsion of the subcomplexes that are enumerated. This line of thought inspired \cite{BBC}, where the authors introduced a class of polynomials arising as a modification of Krushkal and Renhardy's polynomial invariants of triangulations. This last section is devoted to answering a question of \cite{BBC} which we will state after reviewing some definitions (following \cite{BBC,KR}).

\begin{defi}
We denote by $\mathcal S^j$  the family of all \emph{spanning subcomplexes of dimension $j$}, i.e., of all the subcomplexes $S$ such that 
$C^{j-1}\subseteq S \subseteq C^j$. These are naturally identified with the subsets of $(C^j)_j$, the set of $j$-dimensional cells of the $j$-skeleton of $C$. Let $b_i(S)$ be the $i$-th Betti number of $S$ (i.e., the rank of the homology $H_i(S, \mathbb Z)$), and let $t_i(S)$ be the cardinality of its torsion, $t_i(S):=|\operatorname{tor}(H_i(S, \mathbb Z))|$.
\end{defi}

\begin{rem}\label{rem:am} As has been pointed out e.g.\ in \cite{BBGM}, the function $t_i(S)$ is the multiplicity function of the arithmetic matroid defined by the matrix $\partial^j$. 
\end{rem}

\begin{defi}[{\cite[Definition 3.1]{BBC}}]
The $j$-th Tutte--Krushkal--Renhardy (TKR for short) polynomial of $C$ is defined in \cite{KR} as
$$T^j_C(x,y)=\sum_{S\in\mathcal S^j} (x-1)^{b_{j-1}(S)- b_{j-1}(C)}
(y-1)^{b_{j}(S)}.$$

The ``modified $j$-th Tutte--Krushkal--Renhardy (TKR for short) polynomial''  of $C$ is
$$M^j_C(x,y)=\sum_{S\in\mathcal S^j} t_j^2(S) (x-1)^{b_{j-1}(S)- b_{j-1}(C)}
(y-1)^{b_{j}(S)}.$$

\end{defi}

\begin{rem}
Following \cite[Section 5.4]{DKM2} we note that $T^j_C(1,1)$ is the number of "cellular $j$-spanning trees" of $C$ (according to Definition 2.1 of \cite{BBC}), while $M^j_C(1,1)$ is an invariant introduced by G. Kalai in \cite{Ka}: the number of cellular $j$-spanning trees $S$ of $C$, each counted with multiplicity $t_j^2(S)$. Moreover, if $C$ and $C^*$ are dual cell structures on a sphere $\mathbb S^d$ (according e.g.\ to \cite[Definition 1.2]{BBC}) we have 
$$T^j_ {C^*}(x,y)=T^{d-j}_C(y,x)$$
by \cite{KR}, and
$$M^j_ {C^*}(x,y)=M^{d-j}_C(y,x)$$
by \cite[Theorem 3.4]{BBC}.
\end{rem}

In \cite[Remark 3.3]{BBC} the authors ask whether the multiplicity $t_j^2$  defines an arithmetic matroid. The results established in Section 2 allow us to give a positive answer to this question. 

\begin{thm}\label{Thm3} Let $C$ be a CW-complex of dimension $d$ and, for every $j=1,\ldots,d$ let $\mathcal M_j(C)$ denote the cellular matroid of the $j$-skeleton of $C$ (see Remark \ref{rem:cellmat2}). Then, for every $j$ the pair 
($\mathcal M_j(C)$, $t_j^2$) is an arithmetic matroid, and the modified $j$-th Tutte--Krushkal--Renhardy polynomial $M^j_C(x,y)$ is the associated arithmetic Tutte polynomial. In particular, the coefficients of $M^j_C(x,y)$ are nonnnegative.
\end{thm}
\begin{proof}
We first notice that the pair $(\mathcal M_j(C),t_j)$ is an arithmetic matroid (see Remark \ref{rem:am}), then apply Theorem \ref{Thm1} with $m'=m''=t_j$.
\end{proof}

\bibliography{ACWbib}{}
\bibliographystyle{plain}

\end{document}